\documentclass[11pt]{amsart}
\input epsf
\usepackage{latexsym}
\usepackage{amssymb,amsmath,amsthm,amscd,amssymb,enumerate, verbatim}
\usepackage[all]{xy}

\numberwithin{equation}{section}
\newtheorem{cor}[equation]{Corollary}
\newtheorem{lem}[equation]{Lemma}
\newtheorem{prop}[equation]{Proposition}
\newtheorem{thm}[equation]{Theorem}

\newtheorem{quest}[equation]{Question}
\newtheorem{fact}[equation]{Fact}

\newtheorem{Example}[equation]{Example}
\newenvironment{ex}{\begin{Example}\rm}{\end{Example}}
\newtheorem{remark}[equation]{Remark}
\newenvironment{rmk}{\begin{remark}\rm}{\end{remark}}
\def\co{\colon\thinspace}

\newcommand{\Int}{\mbox{Int}}

\newcommand{\Iso}{\mbox{Iso}}

\newcommand{\Min}{\mbox{Min}}

\newcommand{\cd}{\mbox{cd}}
\newcommand{\ab}{\mathrm{ab}}
\newcommand{\e}{\varepsilon}

\def\a{\alpha}
\def\G{\Gamma}
\def\g{\gamma}

\def\b{\beta}

\def\d{\partial}

\def\S1{\bf S^1}

\begin{document}

\title[Obstructions to nonpositive curvature for open manifolds]
{\bf Obstructions to nonpositive curvature for open manifolds}
\thanks{\it 2000 Mathematics Subject classification.{\rm\ 
Primary 53C21, Secondary 57N15, 57S30.}
{\it\ Keywords:\rm\ nonpositive curvature, Busemann function, obstructions}}\rm


\author{Igor Belegradek}

\address{Igor Belegradek\\School of Mathematics\\ Georgia Institute of
Technology\\ Atlanta, GA 30332-0160}\email{ib@math.gatech.edu}


\date{}

\begin{abstract} 
We study algebraic conditions on a group $G$
under which every properly discontinuous,
isometric $G$-action on a
Hadamard manifold has a $G$-invariant 
Busemann function, and for such $G$ we prove the following:
\newline
(1) every open complete nonpositively curved Riemannian 
$K(G,1)$ manifold that is homotopy equivalent to 
a finite complex of codimension $\ge 3$
is an open regular neighborhood of a subcomplex
of the same codimension; (2) each tangential
homotopy type contains infinitely
many open $K(G,1)$ manifolds that admit
no complete nonpositively curved metric even though
their universal cover is the Euclidean space. 
A sample application is that
an open contractible manifold $W$ is homeomorphic
to a Euclidean space if and only if $W\times S^1$
admits a complete Riemannian metric of nonpositive
curvature. 
\end{abstract}

\maketitle

\section{Introduction}

Throughout the paper unless stated otherwise
all manifolds are smooth and connected, 
all metrics are complete
Riemannian, and 
the phrase ``nonpositive sectional curvature'' 
is abbreviated to ``$K\le 0$''.
 
A manifold is {\it covered by $\mathbb R^n$}
if its universal cover is diffeomorphic to $\mathbb R^n$.
The Cartan-Hadamard theorem yields a basic obstruction
to $K\le 0$: any 
$n$-manifold with a metric of $K\le 0$ is covered by $\mathbb R^n$.

Recall that an open $K(G,1)$ manifold covered by 
$\mathbb R^n$ exists if and only if $G$ is a countable
group of finite cohomological dimension.
Countable groups of finite cohomological dimension
form a class of torsion-free groups
that is closed under subgroups,
extensions, and amalgamated products. 

A much deeper obstruction to $K\le 0$ arises 
from Gromov's random, torsion-free, hyperbolic groups with
strong fixed point properties:  
Each of the groups can be realized
as the fundamental group of open manifold covered by 
$\mathbb R^n$ for large enough $n$, yet no such manifold is
homotopy equivalent to a manifold of $K\le 0$, see~\cite{Gro-rand03}
and~\cite[Theorem 1.1]{NaoSil11}. 

Even though the focus of this paper is on open manifolds,
we next review known obstructions to $K\le 0$
for closed manifolds. There exist closed manifolds that are covered by $\mathbb R^n$ and 
have one of the following
properties: their fundamental groups are not 
CAT($0$), see~\cite{Mes90, Dav-book08, Sap11}, 
they are homeomorphic but not diffeomorphic
to manifolds of $K\le 0$~\cite{Ont-double03},
they admit locally CAT($0$) metrics but are
not homotopy equivalent to closed manifold of 
$K\le 0$~\cite{DJL12}.

A manifold is {\it covered by $\mathbb R\times\mathbb R^{n-1}$}
if it is diffeomorphic to the product of $\mathbb R$ and
a manifold covered by $\mathbb R^{n-1}$.
In the paper we describe a sizable class of groups
such that for each group $G$ in the class 
\begin{enumerate}
\item[(i)]
every $K(G,1)$ manifold of $K\le 0$ 
is covered by $\mathbb R\times\mathbb R^{n-1}$;
\vspace{5pt}
\item[(ii)]
there are 
$K(G,1)$ manifolds that are  covered by $\mathbb R^n$ but not
homeomorphic to manifolds covered by 
$\mathbb R\times\mathbb R^{n-1}$.
\end{enumerate}

This yields a homeomorphism obstruction to $K\le 0$ for 
if $G$ satisfies (i)-(ii), then no manifold in (ii) 
is homeomorphic
to a manifold of $K\le 0$.

Manifolds covered by 
$\mathbb R\times\mathbb R^{n-1}$ arise in the context
of nonpositive curvature, namely,
consider a manifold of $K\le 0$ written as $X/G$ where 
$G$ is the deck transformations
group acting isometrically on the universal cover $X$. 
If $G$ {\it fixes a Busemann function\,} (i.e. $G$  
fixes a point at infinity and stabilizes each horosphere
centered at the point), then  $X/G$ is  
covered by $\mathbb R\times\mathbb R^{n-1}$.
More precisely,
any Busemann function is a $C^2$-submersion~\cite{HeiImH77},
so it defines a $C^1$ diffeomorphism of
$X$ and $\mathbb R\times\mathbb R^{n-1}$, 
and if the Busemann function is $G$-invariant,
the $C^1$-diffeomorphism descends to quotients,
which are therefore $C^\infty$-diffeomorphic.

The property ``$G$ fixes a Busemann function'' can
be forced by purely algebraic assumptions on $G$ 
summarized in Sections~\ref{sec: center}-\ref{sec: grounded},
where we also describe various classes of groups
with this property some of which are listed below.

\begin{ex}
\label{ex-intro: busemann list}
Any of the following conditions implies that each
isometric, properly discontinuous $G$-action on 
a Hadamard manifold $X$ fixes a Busemann function:
\begin{itemize}
\item[(1)]
$G=\pi_1(E)$ where $E$ is 
an aspherical manifold that has a 
finitely-sheeted cover which is the total space of
an oriented circle bundle with nonzero rational
Euler class. This applies when $E$ is an 
closed infranilmanifold of nilpotency degree $\ge 2$, or 
a closed $3$-manifold modelled on 
$\widetilde{SL}_2(\mathbb R)$.
\item[(2)]
$G=\pi_1(E)\times H$  where $E$ is as in (1)
and $H$ is any group. 
\item[(3)]
$G=\mathbb Z^n\rtimes H$ where $n\ge 3$ and
$H$ is a finite index subgroup of 
$GL_n(\mathbb Z)$, and the semidirect product
is defined via the standard $H$-action on $\mathbb Z^n$.
\item[(4)] 
$G$ is any group with infinite center
and finite abelianization, e.g. the preimage
of any lattice under the universal cover
$\widetilde{Sp}_{2n}(\mathbb R)\to Sp_{2n}(\mathbb R)$, 
$n\ge 2$; actually this example is also covered by (1).
\item[(5)]
$G$ is the amalgamated product along $\mathbb Z^n$
of the group $\mathbb Z^n\rtimes H$ as in (4)
and any group $K$ that contains $\mathbb Z^n$ as a normal
subgroup. 
\item[(6)]
$G$ has infinitely generated center, e.g.
the product of $(\mathbb Q, +)$ and any
group.
\end{itemize}
\end{ex}

In Section~\ref{sec: visib} we show that if 
the Hadamard manifold $X$ is {\it visibility\,} 
(i.e. any two distinct points at infinity 
are the endpoints of a geodesic in $X$), then
the class of groups that fix Busemann functions on $X$
gets even larger, e.g. it contains any nontrivial
product $G=G_1\times G_2$ such that there is a 
$K(G, 1)$ manifold.

An essential tool in understanding  
manifolds covered by $\mathbb R\times\mathbb R^{n-1}$ 
is the recent result of
Guilbault~\cite{Gui-prod-R07}: if an open manifold $W$
of dimension $\ge 5$ is homotopy equivalent to a finite
complex, then $\mathbb R\times W$ is diffeomorphic
to the interior of a compact manifold.
As an addendum to~\cite{Gui-prod-R07} we show: 

\begin{thm} 
\label{thm-intro: reg nbhd}
Let $W$ be an open $(n-1)$-manifold with $n\ge 5$ 
that is homotopy equivalent
to a finite complex of dimension $k\le n-3$.
Then $\mathbb R\times W$ is diffeomorphic to
the interior of a regular neighborhood of a $k\!$-dimensional
finite subcomplex. 
\end{thm}

In addition to fixing a Busemann function 
there are other properties that 
can force any open $K(G,1)$ manifold of $K\le 0$
to be the interior of a regular neighborhood of
a finite subcomplex. 

\begin{ex}
\label{intro-ex: flat} 
If $G$ be the fundamental group of a closed flat manifold, 
then for any $K(G,1)$ manifold with $K\le 0$,
the $G$-action on the universal cover $X$
either fixes a Busemann function or
stabilizes a flat $F$ where it acts cocompactly; 
in the latter case
$X/G$ is diffeomorphic to the normal bundle of the closed
flat submanifold $F/G$, see Lemma~\ref{lem: center has parabolic}
and~\cite[Corollary II.7.2]{BriHae}.
\end{ex}

\begin{ex}
\label{intro-ex: higher rank} 
Let $G$ be an irreducible torsion-free uniform lattice in a 
connected higher rank semisimple Lie group with finite center and no compact factor. 
Then harmonic maps superrigidity~\cite[Theorem 1.2]{Duc} 
implies that for any $K(G,1)$ manifold with $K\le 0$,
the $G$-action on the universal cover $X$
either acts cocompactly on a totally geodesic
subspace $F$, or fixes a point at infinity in which case 
it also fixes the corresponding Busemann function as 
$G^\ab$ is finite, see Remark~\ref{rmk: no hom into R}. 
\end{ex}

As we show later, Theorem~\ref{thm-intro: reg nbhd} and
Example~\ref{intro-ex: flat} yield 
an attractive characterization of $\mathbb R^n$:

\begin{cor} 
\label{cor-intro: Rn}
An open contractible $n\!$-manifold $W$ is homeomorphic
to $\mathbb R^n$ if and only if $W\times S^1$
admits a metric of $K\le 0$. 
\end{cor}

If $n=4$, then ``homeomorphic'' in 
Corollary~\ref{cor-intro: Rn} cannot be replaced with
``diffeomorphic''
because if $W$ is an exotic $\mathbb R^4$, then
$W\times S^1$ is diffeomorphic to 
$\mathbb R^4\times S^1$~\cite[Theorem 2.2]{Sie-collar}.

Since there are only countably many diffeomorphism classes of 
compact manifolds, there are only countably many possibilities
for open regular neighborhoods.
Various classification results for manifolds diffeomorphic
to the interiors of regular neighborhoods 
can be found in~\cite{Hae-emb, RouSan-blk-bunI, Sie-collar, LicSie}
and~\cite[Corollary 11.3.4]{Wal-book}. 
In the stable range, i.e. when the subcomplex has dimension 
$<\frac{n}{2}$, 
the interior of a regular neighborhood of dimension $n$
is determined up to diffeomorphism by its tangential
homotopy type~\cite{LicSie}. 

\begin{cor} Let $G$ be a group as in
\textup{Examples~\ref{ex-intro: busemann list},
\ref{intro-ex: flat}, \ref{intro-ex: higher rank}}.
If $V$ is an open $n$-manifold that
is homotopy equivalent to a finite $K(G,1)$ complex
of dimension $<\frac{n}{2}$, where $n\ge 5$, then 
the tangential homotopy type of $V$ contains 
at most one manifold of $K\le 0$. 
\end{cor}

If an $n$-dimensional regular neighborhood $R$ 
is homotopy equivalent to a closed manifold $B$,
then $\mathrm{Int}(R)$ is the total space of a vector bundle
over $B$, provided $\dim(B)<\frac{2n-2}{3}$ 
and $n>4$~\cite{Hae-emb, Sie-collar}; the former inequality
is known as the metastable range.
Together with Theorem~\ref{thm-intro: reg nbhd} this
implies:

\begin{cor}
Let $B$ be a closed aspherical $k$-manifold as in 
\textup{Example~\ref{ex-intro: busemann list}(1)} or
\textup{Examples~\ref{intro-ex: flat}-\ref{intro-ex: higher rank}}, and let $V$ be an open $n$-manifold
of $K\le 0$ that is homotopy equivalent to $B$. If $2n>3k+2$ and $n>4$,
then $V$ is diffeomorphic to the total space of a vector bundle
over $B$.
\end{cor}

We now turn to methods of producing manifolds as
in (ii). A trivial method
is to consider a $K(G,1)$ manifold
that has the minimal dimension among all $K(G,1)$ manifolds 
that are covered by $\mathbb R^n$, which yields:

\begin{prop}
Any aspherical manifold is homotopy equivalent
to a manifold covered by $\mathbb R^n$ but
not covered by $\mathbb R\times\mathbb R^{n-1}$.
\end{prop}

\begin{cor} If $G$ is as in\, 
\textup{Example~\ref{ex-intro: busemann list}},
then any $K(G,1)$ manifold is homotopy equivalent to a manifold 
that admits no metric of $K\le 0$
and is covered by a Euclidean space.
\end{cor}

This trivial method always works, yet it is non-constructive
for it is not easy to decide whether a specific
open manifold has the minimal dimension in the above sense,
see~\cite{BKK02, BesFei02, Yoo04, Des06}. 

On the other hand, if $G$ has a finite $K(G,1)$ complex, 
we produce explicit manifolds satisfying (ii):

\begin{thm}
\label{thm-intro: count many} 
Let $V$ be an open aspherical $n$-manifold that is
homotopy equivalent to a finite complex
of dimension $\le n-3$, where $n\ge 5$. 
Then the tangential homotopy type of
$V$ contains countably many manifolds 
such that each of them is covered by $\mathbb R^n$ and
is not homeomorphic to 
the interior of a regular neighborhood
of a finite subcomplex of codimension $\ge 3$, 
and in particular not covered by
$\mathbb R\times\mathbb R^{n-1}$.
\end{thm}

It should be possible to strengthen the conclusion 
``countably many'' in Theorem~\ref{thm-intro: count many} 
to ``a continuum of'', but we only prove this
under a technical assumption 
which holds in many cases of interest,
including when $\pi_1(V)$ is a Poincar{\'e} duality group,
or when $\pi_1(V)$ does not contain a subgroup
isomorphic to $\mathbb Z^3$,
see Section~\ref{sec: without R-factors}.

The structure of the paper is as follows:
Sections~\ref{sec: center},
\ref{sec: grounded}, \ref{sec: visib}
discuss when a group fixes a Busemann function,
and in particular, justify 
Examples~\ref{ex-intro: busemann list},
\ref{intro-ex: flat}, \ref{intro-ex: higher rank}.
In Section~\ref{sec: with R-factors} we study manifolds
with $\mathbb R$-factors, and prove
Theorem~\ref{thm-intro: reg nbhd} and 
Corollary~\ref{cor-intro: Rn}.
In Section~\ref{sec: without R-factors}
we construct manifolds without $\mathbb R$-factors,
and prove Theorem~\ref{thm-intro: count many}
as well as its harder uncountable version. 
The latter requires a certain family of homology 
$3$-spheres described in Appendix~\ref{sec: homology sph}.
Throughout the paper we make implicit use of
various finiteness results for groups, which are
recalled in Appendix~\ref{sec: app on finitenss}.

{\bf Acknowledgments.} 
The author is grateful for NSF support (DMS-1105045)
and to MathOverflow for a stimulating work environment.
Thanks are due to Karel Dekimpe for 
Examples~\ref{ex: dekimpe 3d} and~\ref{ex: dekimpe poly},
Misha Kapovich for Example~\ref{ex: Kapovich ex},
Andy Putman and Wilberd van der Kallen for
help with Proposition~\ref{prop: Putman ex}.

\section{Parabolics in the center and invariant Busemann functions}
\label{sec: center}

Background references for nonpositive curvature
are~\cite{BGS, Ebe-book, BriHae}. 
In this section $G$ is a discrete isometry group 
of a Hadamard manifold $X$. 

\begin{lem}
\label{lem: periodic group}
If $W$ is a closed convex subset of $X$, then 
any abelian discrete subgroup $A\le\Iso(W)$ that contains no parabolic isometries is finitely
generated of rank $\le\dim(W)$.
\end{lem}
\begin{proof} 
By~\cite[lemma 7.1(1)]{BGS} 
the set $\bigcap_{g\in A}\Min(g)$ is a closed
$A$-invariant convex subset of $M$ which splits as
$C\times\mathbb R^k$, where $A$ acts trivially on
the $C$-factor, and by translations on $\mathbb R^k$.
Since $A$ is discrete, it acts freely on $\mathbb R^k$
so $A\cong \mathbb Z^m$ for $m\le k$.
\end{proof}

\begin{lem} 
\label{lem: center has parabolic}
$G$ preserves a Busemann function
if it has a finite index subgroup $G_0$ whose center $Z(G_0)$
contains a parabolic isometry.
\end{lem}
\begin{proof}
This result is implicit in~\cite[Lemma 7.3, 7.8]{BGS},
but for completeness we record a proof below.
Let $z\in Z(G_0)$ be a parabolic isometry, 
and $g_1,\dots g_k$ are right coset representatives
of $G_0$ in $G$, then the 
function $x\to \sum_i d(g_i^{-1}zg_i x, x)$
is convex, and $G$-invariant because $d(g_i^{-1}zg_i x, x)$
is independent of a coset representative in $G_0 g_i$ as 
$z\in Z(G_0)$.
Since $x\to d(zx, x)$ does not assume its infimum,
neither does the above convex function. Then
a limiting process outlined in~\cite[Lemma 3.9]{BGS}, 
and fully explained in~\cite[Lemma II.8.26]{BriHae},
gives rise via Arzela-Ascoli theorem to a $G$-invariant
Busemann function. 
Note that this Busemann function is
not explicit and there seems to be no way to find
a $G$-invariant Busemann function associated with 
a given point
of $X(\infty)$. 
\end{proof}

Flat Torus Theorem~\cite[Chapter II.7]{BriHae} implies that
a discrete isometry group of $X$ whose center consists of axial
isometries is quite special, which for our purposes
is summarized as follows.

\begin{thm}\label{thm: into-center}
Let $G$ be a group with subgroups $H$, $G_0$ such that
their centers $Z(H)$, $Z(G_0)$ are infinite, 
$Z(H)\subseteq Z(G_0)$, the index of $G_0$ in $G$
is finite, and one of the following conditions hold: \newline
\textup{(1)} $Z(H)$ is not finitely generated;\newline
\textup{(2)} any homomorphism $H\to\mathbb R$ is trivial.\newline
\textup{(3)} $H$ is finitely generated, and
$Z(H)$ contains a free abelian subgroup
that is \newline\phantom{\textup{(3)}} 
not a direct factor of any 
finite index subgroup of $H$.\newline
Then any effective, properly discontinuous, 
isometric $G$-action
on a Hadamard manifold fixes a Busemann function.
\end{thm}

The reader may want to first think through 
the case when $H=G_0=G$, and then go on
to observe that if Theorem~\ref{thm: into-center}
applies to $H$, $G_0$, $G$, then it also 
applies to $H$, $G_0\times K$, $G\times K$
for any group $K$.

\begin{proof}
The result follows by Lemma~\ref{lem: center has parabolic}
once we show that $Z(G_0)$ contains a parabolic isometry.
Since $Z(H)\subseteq Z(G_0)$ it suffices to find
a parabolic element in $Z(H)$. Suppose no such element exists. 
Then $Z(H)$ is finitely generated by 
Lemma~\ref{lem: periodic group}, which implies (1).

To prove (2) note that if $H$
has no nontrivial homomophism into $\mathbb R$,
then by~\cite[Remark II.6.13]{BriHae}
$Z(H)$ contains no axial isometries, so 
it is a finitely generated, abelian, and periodic,
and hence finite which contradicts the assumptions.

Finally, (3) is deduced as follows. Since $Z(H)$ is finitely
generated and infinite, it has a free abelian 
subgroup $A$ of positive rank. 
Under the assumption that $H$ is finitely generated, 
\cite[Theorem II.7.1]{BriHae} implies that
$H$ has a finite index subgroup that contains $A$
as a direct factor.
\end{proof}

\begin{ex}
Theorem~\ref{thm: into-center}(1) applies, 
e.g. to any infinitely generated, torsion-free, countable
abelian group of finite rank, such as $(\mathbb Q, +)$,
where finiteness of rank ensures finiteness of cohomological 
dimension~\cite[Theorem 7.10]{Bie-book}.
\end{ex}

\begin{quest}
Is there a group of type {\it F} with infinitely generated
center?
\end{quest}

\begin{rmk}
There exists a finitely presented group with solvable word problem
whose center contains every countable abelian group~\cite[Corollary 3]{Oul-center}.
On the other hand, if $H$ is a group with infinitely generated center, then
$H$ cannot be embedded into hyperbolic or CAT($0$) groups
because their abelian subgroups
are finitely generated. If in addition $H$ has type 
{\it F}, or even of type {\it FP}, then $H$ 
cannot be linear over a field of characteristic 
zero~\cite[Corollary 5]{AlpSha82}, 
and cannot be elementary amenable, as we explain below
for completeness.
\end{rmk}

\begin{prop}
Every elementary amenable group of type {\it FP} has finitely generated center.
\end{prop}
\begin{proof}
Since $G$ has type {\it FP}, its cohomological dimension is finite.
Any elementary amenable of finite cohomological dimension is virtually solvable~\cite{HilLin92}
and of course torsion-free. As $G$ has type 
{\it FP}, so does
the solvable finite index subgroup $G_0$ of 
$G$~\cite[Theorem VIII.6.6]{Bro-book}. 
Solvable groups of type {\it FP},
are constructable~\cite{Kro86}
and hence quotients of $G_0$ are finitely 
presented~\cite[Theorem 4]{BauBie76}.
So $Z(G_0)$ is finitely generated, for
if the quotient of a finitely generated group
by the center is finitely presented, then the center is 
finitely generated~\cite[Corollary 1, page 54]{Bau-book}.
Hence the subgroup $Z(G)\cap G_0$ of $Z(G_0)$
is finitely generated, and therefore the same holds
for $Z(G)$.
\end{proof}

\begin{ex}
Here is an example of a group of type {\it F}
with infinite center
and finite abelianization (to which 
Theorem~\ref{thm: into-center}(2) applies).
Let $G$ be the preimage
of a lattice $\Gamma\le Sp_{2n}(\mathbb R)$ under the universal 
cover, then $Z(G)$ is infinite,
and $G$ has Kazhdan property (T) when 
$n\ge 2$~\cite[Example 1.7.13]{BHV},
and hence finite abelianization.
If $\Gamma$ is torsion-free, then it has type 
{\it F}~\cite{BorSer74}, and hence so does $G$,
see Appendix~\ref{sec: app on finitenss}.
Note that Theorem~\ref{thm: into-center}(3) also 
applies to $G$ because its finite index subgroups
have property (T) and hence finite abelianization
which implies that they cannot
have $\mathbb Z$ as a direct factor.	
\end{ex}

Theorem~\ref{thm: into-center}(3) also applies
when $G$ is virtually nilpotent, finitely generated,
and not virtually abelian, namely, we let $H=G_0$ be
a nilpotent torsion-free subgroup of finite index, and
use the following.

\begin{prop} 
\label{prop: nil not virt split}
If $H$ is a 
torsion-free, nonabelian,
nilpotent group $H$,
then $Z(H)$ is not a virtual direct factor of $H$.
\end{prop}
\begin{proof}
Arguing by contradiction suppose $H_0$ is a finite 
index subgroup of $H$ that contains $Z(H)$
as a direct factor, i.e. there is $K\mathrel{\unlhd} H_0$
with $H=KZ(H)$ and $K\bigcap Z(H)=1$.
Then $Z(H_0)=Z(K)Z(H)$ where $Z(K)\bigcap Z(H)=1$.
A key observation is that 
$Z(H_0)\subseteq Z(H)$: if we take $z\in Z(H_0)$
and $g\in H$, and pick $k$ with $g^k\in H_0$, then
$g^k=(zgz^{-1})^k$ but in torsion-free nilpotent
groups roots are unique~\cite[Theorem 16.2.8]{KarMer}, 
so $g=zgz^{-1}$, and hence $z\in Z(H)$. 
It follows that $Z(K)$ is trivial, so since $K$
is nilpotent we conclude that $K$ is trivial, i.e.
$H_0=Z(H)$. Now if $g,h\in H$, then $h^k\in H_0$
so $gh^kg^{-1}=h^k$ and again uniqueness of roots
implies  $ghg^{-1}=h$ so that $H$ is abelian.
\end{proof}

It is a well-known fact (which we will not use, 
and hence do not prove) that the group $H$ in 
Proposition~\ref{prop: nil not virt split}
can be realized as the fundamental group
of the circle bundle over a closed nilmanifold, and
noncommutativity of $H$ allows one to choose 
the bundle so that its rational Euler class is nonzero.
Thus Proposition~\ref{prop: nil not virt split}
is also implied by the proposition below.

\begin{prop}
\label{prop: euler class}
If $B$ is an aspherical manifold, and
$E\to B$ is a principal circle bundle
with nonzero rational Euler class,
then the fiber circle is a nontrivial element $s$
in the center of $\pi_1(E)$
that is not a virtual direct factor.
\end{prop}
\begin{proof}
The circle action on $E$ lifts to any cover
$\hat E$ whose fundamental group contains 
$s$~\cite[Theorem I.9.3]{Bre-book}. Therefore,
the covering $\hat E\to E$ descends to a covering of base
spaces $\hat B\to B$. It follows that $\hat E\to \hat B$
is a pullback of $E\to B$ via $\hat B\to B$.
If $\hat E\to \hat B$ is a finite cover, then
it has a nonzero rational Euler class
because finite covers induce injective maps on 
rational cohomology (due to existence of a transfer).
Finally, if $s$ were a 
direct factor in $\pi_1(\hat E)$, then 
$\hat E\to \hat B$ would have a section, so
the Euler class would be zero. 
\end{proof}

\begin{rmk} Theorem~\ref{thm: into-center}(3)
also applies to the fundamental group of any 
closed Seifert fibered space modelled
on $\widetilde{SL}_2(\mathbb R)$, because it
has a finite index subgroup to which
Proposition~\ref{prop: euler class} applies.
\end{rmk}

Further examples of groups that satisfy 
Theorem~\ref{thm: into-center}(3) 
can be obtained by amalgamating
one such example with any group along a common central
subgroup, which uses the following fact.

\begin{fact}
Let $G:=G_1\ast_A G_2$ where $A$ lies in the center
of $G_1$, $G_2$. If $B\le A$ is a virtual
direct factor of $G$, then $B$ is a virtual
direct factor of $G_1$ and $G_2$.
\end{fact}
\begin{proof} Note that $B\le Z(G)$.
If $K$ is a finite index subgroup
of $G$ of the form $BL$, where $L$ is normal in $K$
and $B\cap L=\{1\}$, then $K\cap G_i$ is a finite
index subgroup of $G_i$ of the form $B(L\cap G_i)$
where $L\cap G_i$ is normal in $K\cap G_i$ and
$B\cap (L\cap G_i)=\{1\}$, so $B$ is a virtual direct factor of $G_i$. 
\end{proof}

\section{Grounded groups and invariant Busemann functions}
\label{sec: grounded}

In this section $G$ is a discrete isometry group 
of a Hadamard manifold $X$. 

If $G$ fixes a point at infinity, then $G$
permutes Busemann functions (or equivalently,
permutes horospheres) associated with the point.

\begin{ex}
An axial isometry $z\to 2z$ of the upper half permutes
horospheres centered at $0$, as well as horospheres
centered at $\infty$.
\end{ex}

\begin{fact} If $G$ fixes a point $[b]$ at infinity
represented by a Busemann function $b$, then
the map $g\to b(g x)-b(x)$ is a group homomorphism
$G\to\mathbb (\mathbb R, +)$ 
that depends only on $[b]$, and not on $x$
or $b$.
\end{fact}
\begin{proof} 
This is immediate from the cocycle property
of the Busemann function, e.g. see the discussion 
after~\cite[Proposition 3.11]{CapMon09}
where the above homomorphism was called
the {\it Busemann character}.
\end{proof}

\begin{rmk}
\label{rmk: no hom into R}
Thus if $G\le\Iso(X)$ has no nontrivial homomorphism
into $\mathbb R$, then
$G$ preserves a Busemann function if and only if 
$G$ fixes a point at infinity. 
\end{rmk}

To further enlarge the class of groups that preserve
Busemann functions we discuss, following 
Schroeder in~\cite[Appendix 3]{BGS} and
Eberlein~\cite[Section 4.4]{Ebe-book}, 
groups which we call {\it grounded} (and which were
called {\it admissible} in~\cite{Ebe-book}).
Namely, equip the ideal boundary $X(\infty)$ of a
Hadamard manifold $X$ with the 
Tits metric, which is $CAT(1)$
and complete~\cite[Theorem II.9.20]{BriHae}, and
call a subgroup $G\le\Iso(X)$ {\it grounded\,} if
its fixed point set at infinity
$X_G(\infty)$ is a nonempty subset of (intrinsic) radius 
$\le\frac{\pi}{2}$, i.e. $X_G(\infty)$ lies
in the closed $\frac{\pi}{2}$-ball centered at a point
of $X_G(\infty)$.

\begin{ex}
There are discrete cocompact subgroups of $\Iso(\mathbb R^n)$ with finite
abelianizations, such as $(\mathbb Z_2\ast\mathbb Z_2)^n$, or
see~\cite{Szc85} for torsion-free examples.
No such group can fix a point at infinity of $\mathbb R^n$
because no lattice fixes a Busemann function. By contrast,
any group of translations, such as the standard $\mathbb Z^n$, fixes the ideal boundary of $\mathbb R^n$ pointwise, yet 
it is not grounded as its Tits boundary is isometric
to the standard unit sphere, which has radius $\pi$.
\end{ex}

\begin{thm} 
\label{thm: intro-grounded}
A discrete, infinite isometry group $G$ of a Hadamard manifold
is grounded if one of the following is true:
\newline
\textup{(1)}
$G$ is as in \textup{Theorem~\ref{thm: into-center}}.
\newline
\textup{(2)}
$G$ has a grounded normal subgroup.
\newline
\textup{(3)}
$G$ is virtually solvable and not virtually-$\mathbb Z^k$ for any $k$.
\newline
\textup{(4)}
$G$ is the union of a nested sequence of grounded subgroups.
\newline
\textup{(5)}
a normal abelian subgroup of $G$ 
contains an infinite $G$-conjugacy class.
\end{thm}
\begin{proof}
Schroeder proved in~\cite[Appendix 3]{BGS} the key fact 
that an abelian group that contains a parabolic is grounded,
see also~\cite[Section 4.4]{Ebe-book}. These references
also contain a proof of part (2) which we outline below. 
If $N$ is a grounded normal subgroup of $G$, then
$G$ stabilizes $X_N(\infty)$, and
hence fixes its center of gravity, which proves
that $X_G(\infty)$ is nonempty. It is clear that
$X_G(\infty)\subseteq X_N(\infty)$ and that $X_G(\infty)$
is closed in the cone topology on $X(\infty)$.
Since the identity map of $X(\infty)$ from the Tits metric 
topology to the cone the topology is continuous, $X_G(\infty)$
is closed in the Tits metric topology. 
Finally, convexity of $X_G(\infty)$ is implied
by the fact that the $\frac{\pi}{2}$-balls are uniquely geodesic,
so if an element of $G$ fixes the endpoints of a geodesic segment,
it fixes the segment pointwise. 

To prove (4) let $G=\bigcup_i G_i$ where 
$G_1\le \dots\le G_i\le \dots$ is a countable family of 
grounded subgroups of $\Iso(X)$ (since $G$ is countable,
it is not the union of uncountable family of distinct groups). 
Note that $X_G(\infty)=\bigcap_i X_{G_i}(\infty)$.
By assumptions $X_{G_i}(\infty)$ is a nested family
of closed convex subsets of radius $\le\frac{\pi}{2}$.
It is shown by Caprace-Lytchak~\cite[Lemma 5.2]{CapLyt10}
that any such family must have a nonempty intersection
provided the ambient metric space is finite dimensional,
which follows from~\cite[Proposition 2.1]{CapLyt10}
because $X$ is finite dimensional. 
Thus $X_G(\infty)$ is a nonempty,
closed, convex subset of radius 
$\le\frac{\pi}{2}$, so $G$ is grounded.

Now we are ready to prove (1). The conclusion would
be immediate if $G_0$ were normal in $G$ because
by the proof of Theorem~\ref{thm: into-center}
$Z(G_0)$ contains a parabolic element, so
$G_0$, and hence $G$ is grounded by the above-mentioned result
of Schroeder. 
In general, since $G$ is countable,
we can realize $G$ as the union of a nested
sequence of finitely generated subgroups $\{H_i\}$
where $H_i\subset H_{i+1}$ for all $i$. We may also
assume that $H_1$ contains a parabolic element
$z\in Z(G_0)$. 
The subgroup $N_i:=\bigcap_{g\in H_i} g(H_i\cap G_0)g^{-1}$
is normal in $H_i$, and moreover, the index of $N_i$
in $H_i$ is finite because $H_i$ is a finitely generated,
and the index of $H_i\cap G_0$ is finite.
Hence $N_i$ contains a power of $z$, which
is a parabolic element in $Z(N_i)$. It follows that
$N_i$ is grounded, and hence so is $H_i$.
Now (4) implies that $G$ is grounded as the union of $H_i$'s.

To prove (3) we first recall a result of 
Burger-Schroeder~\cite{BurSch87} that
an amenable subgroup of $\Iso(X)$ either
stabilizes a flat, or fixes a point of $X(\infty)$.
In the former case the subgroup is 
virtually-$\mathbb Z^k$, which is excluded by assumption.
(If a discrete subgroup $G\le\Iso(X)$ stabilizes a $m$-dimensional flat,
then restricting the $G$-action to the flat gives a
homomorphism $G\to \Iso(\mathbb R^m)$ with
finite kernel and discrete image. By Bieberbach's theorem 
the image of $G$ stabilizes an 
$k$-dimensional flat in $\mathbb R^k$,
which implies that $G$ is 
virtually-$\mathbb Z^k$, see e.g~\cite[Theorem 2.1]{Wil00}).

Thus we may assume that $X_G(\infty)$ is nonempty.
By Lemma~\ref{lem: periodic group} abelian subgroups of $G$ are finitely generated, so
$G$ is virtually polycyclic as
proved by Mal'cev~\cite{Mal51}. 
A polycyclic subgroup of $\Iso(X)$ that contains no parabolics
must be virtually abelian~\cite[Theorem II.7.16]{BriHae}.
Thus $G$ contains a parabolic element. If $P$ is the cyclic subgroup of $G$ generated by a parabolic element, then 
$X_G(\infty)\subseteq X_P(\infty)$,
where $X_P(\infty)$ has radius $\le\frac{\pi}{2}$
because $P$ is grounded.

To prove (5) consider an normal abelian subgroup 
$A$ of $G$ that contains an infinite $G$-conjugacy class.
Then $A$ contains a parabolic 
element~\cite[Lemma II.7.17(2)]{BriHae}, which
implies that $A$, and hence $G$ are grounded.
\end{proof}

\begin{cor}
\label{cor: elem amen}
Let $G$ be a finitely generated,
torsion-free, discrete subgroup of $\Iso(X)$.
If $G$ has a nontrivial, normal, elementary amenable subgroup, then
either $G$ is grounded, or $G$ has a nontrivial,
finitely generated, 
abelian, normal subgroup that is  
a virtual direct factor of $G$.
\end{cor}
\begin{proof}
Since $G$ is torsion-free and discrete, it
has finite cohomological dimension, so 
by~\cite[Corollary 2]{HilLin92}
$G$ has a nontrivial, abelian, normal subgroup $A$.
If $A$ contains a parabolic element, then $G$ is grounded. 
Otherwise, by Lemma~\ref{lem: periodic group} 
$A$ is isomorphic to some $\mathbb Z^k$, $k>0$, and 
then~\cite[Theorem II.7.1(5)]{BriHae} implies that $A$ is 
a virtual direct factor of $G$. 
\end{proof}

\begin{rmk} 
If $G$ is an amenable, discrete, torsion-free 
isometry group of $X$
that is not virtually-$\mathbb Z^k$ for any $k$,
then $G$ fixes a point at infinity~\cite{BurSch87}, 
yet I do not know whether it is always grounded.
A counterexample, if it exists, would have to be
an amenable group of finite cohomological dimension
that is not elementary amenable, for if $G$ were
elementary amenable, it would be
virtually solvable~\cite{HilLin92},
and hence grounded by 
Theorem~\ref{thm: intro-grounded}(3). 
\end{rmk}

In view of Remark~\ref{rmk: no hom into R} 
one wants to produce grounded groups with
no nontrivial homomorphisms into $\mathbb R$.

\begin{ex} 
\label{ex: dekimpe 3d}
A complete list
of $3$-dimensional closed 
infranilmanifolds, as found in~\cite{DIKL-3dnil},
contains an infinite family of manifolds 
whose fundamental group has finite abelianization.
Namely, in the family 2 on~\cite[page 156]{DIKL-3dnil},
also listed on~\cite[page 160]{Dek-book}, let
$k_1$ be a positive even integer, and set
$k_2=0=k_3$, $k_4=1$. Taking products
yields examples in each dimension
divisible by $3$. 
This example does not really illustrate the power
of Theorem~\ref{thm: intro-grounded}(3)
because it is already covered by Theorem~\ref{thm: into-center},
yet it can be used as a building block in other examples below.
For that purpose we note that any group in the 
above family surjects onto the infinite dihedral group.
\end{ex}

The following virtually solvable group 
has finite abelianization,
finite cohomological dimension, 
and is not virtually polycyclic, 
so it satisfies assumptions of
Theorem~\ref{thm: intro-grounded}(3).

\begin{ex} 
\label{ex: Kapovich ex}
(Misha Kapovich)
Let $H$ be a virtually polycyclic group of finite cohomological dimension such that $H^{\ab}$ is a finite group that
has ${\mathbb Z}_m$ among its cyclic factors for even $m$  
(e.g. we could take $H$ to be
the fundamental group of Hantsche-Wendt's flat manifold, or 
any group in Example~\ref{ex: dekimpe 3d}). 
Set $B:=\mathbb Q\rtimes{\mathbb Z}$, where ${\mathbb Z}$ acts on 
$\mathbb Q$  by $q\to -q$ via the surjection onto ${\mathbb Z}_2$. 
Set $G:=B\rtimes H$, where
$H$ acts on $B$ so that the action on $\mathbb Q$ is trivial and the action on ${\mathbb Z}$ is by $m\to -m$ (which is indeed an action as $\mathbb Z$ is abelian). By construction $G^{\ab}$ is finite, and $G$ has finite cohomological dimension because its semidirect factors have this property, 
e.g. $\mathbb Q$ is locally cyclic. 
Finally, $G$ is not virtually polycyclic as it contains 
$\mathbb Q$, and abelian 
subgroups of virtually polycyclic groups are 
finitely generated.
\end{ex}

In the above example $G$ is not of type {\it F}
(because it has $\mathbb Q$ as a virtual direct factor).
The following is an example of a torsion-free polycyclic
group with finite abelianization which is not 
virtually nilpotent, and like every virtually
polycyclic group,
it is the fundamental group of a closed aspherical 
manifold~\cite{AusJoh}.

\begin{ex} 
\label{ex: dekimpe poly}
(Karel Dekimpe)
Let $H$ be any virtually polycyclic group
with finite abelianization such that $H$ surjects
onto the infinite dihedral group 
$D_\infty=\langle a,b| b^2=1, aba=b\rangle$.
Let 
\begin{equation*}
\a=\left[
\begin{array}{cc}
Q & 0 \\
0 & Q^{-1} 
\end{array}\right]
\quad\quad\quad
\b=\left[
\begin{array}{cc}
0 & I_n \\
I_n & 0 
\end{array}\right]
\end{equation*}
be matrices on $GL_{2n}(\mathbb Z)$, where 
$Q\in GL_{n}(\mathbb Z)$. Since $\b^2=I_{2n}$ and
$\a\b\a=\b$, the map that sends
$a$, $b$ to $\a$, $\b$, respectively,
extends to a homomorphism
$D_\infty\to\ GL_{2n}(\mathbb Z)$.
It remains to choose $Q$ so that 
the corresponding semidirect product
$G:=\mathbb Z^4\rtimes H$ has finite abelianization,
and is not virtually nilpotent. 
The abelianization of the semidirect product
$A\rtimes B$ is $(A^\ab)_B\times B^\ab$, where
$(A^\ab)_B$ is the coinvariants for the $B$-action 
on $A^\ab$, i.e. the quotient of $A$ by the subgroup 
generated by the set 
$\{z-\overline{\phi}_b(z)| z\in A^{\ab}, b\in B\}$,
where $\bar{\phi}$ is $\phi$ followed by the projection
$A\to A^{\ab}$. Take $Q\in SL_2(\mathbb R)$ with
$\mathrm{trace}(Q)\neq 2$ so that $A-I_2$ has nonzero determinant.
By considering where $\a$, $\b$ map the standard generators
of $\mathbb Z^4$ one gets finiteness of the coivariants. 
To show that $G$ is not virtually nilpotent, we assume in addition that
$|\mathrm{trace}(Q)|> 2$, and consider the subgroup $H_0\le H$
that is the preimage of the subgroup 
$\langle a\rangle\le D_\infty$, and
note the $\mathbb Z^2$-summand
of $\mathbb Z^4$ spanned by the first two basis
vectors is $H_0$-invariant. Thus $G$ has a subgroup
$\mathbb Z^2\rtimes H_0$, which
surjects onto $\mathbb Z^2\rtimes \langle a\rangle$.
With our choice of $Q$ the latter group is a lattice in SOL, 
a $3$-dimensional solvable Lie group. Since lattices in
SOL are not virtually nilpotent, neither is $G$, as being virtually nilpotent
is inherited by subgroups and quotients.
\end{ex}

The following gives a large family of grounded groups
of the form $\mathbb Z^k\rtimes K$, which have
finite abelianization and type {\it F\,} if $K$ does.
This applies e.g. if $K$ is a torsion-free, finite index subgroup
of $GL_k(\mathbb Z)$ where $k\ge 3$ and the semidirect
product is defined via its standard action on $\mathbb Z^k$.

\begin{prop}
\label{prop: grounded semidirect}\rm
Theorem~\ref{thm: intro-grounded}(5)\it\
applies to every semidirect product
$G=\mathbb Z^k\rtimes K$ defined by a homomorphism
$\phi\co K\to GL_k(\mathbb Z)$ whose image contains an infinite order element. 
\end{prop}
\begin{proof}
If $h\in H$ is mapped to an infinite order element
$\phi_h$, then there is $v\in\mathbb Z^k$ such that
$\phi_{h^m}(v)\neq v$ for $m\neq 0$ (e.g. 
$v$ can be chosen as a vector of the standard basis $\{e_i\}$, 
for otherwise for each $i$ we have
$\phi_{h^{l_i}}(e_i)=e_i$ for some $l_i$, so if
$l=l_1\dots l_k$, then by linearity
$\phi_{h^{l}}$ is the identity, contradicting the
infinite order assumption). 
Thus $h^m v h^{-m}=\phi_{h^m}(v)\neq v$ in $G$ 
for all $m\neq 0$ so the conjugacy class of $a$
intersects $\mathbb Z^k$ along an infinite subset, and
Theorem~\ref{thm: intro-grounded}(5) applies.
\end{proof}

\begin{prop}
\label{prop: Putman ex}
In the notations of {\rm
Proposition~\ref{prop: grounded semidirect}}
if $\phi(K)$ is a finite index subgroup of 
$GL_k(\mathbb Z)$, and if $K^\ab$ is finite,
then $G^\ab$ is finite. In particular,
$G^\ab$ is finite if $\phi$ is injective and $n\ge 3$. 
\end{prop}
\begin{proof}
As in Example~\ref{ex: dekimpe poly},
the abelianization of 
$A\rtimes B$ is $(A^\ab)_B\times B^\ab$.
Finiteness of $(\mathbb Z^k)_K$ follows via the 
elementary computation below shown to 
me by Wilberd van der Kallen.
For $i\neq j$, let $E_{ij}^m\in GL_k(\mathbb Z)$ be 
the matrix with $1$'s on the diagonal, $m$
as $ij$-entry, and zero elsewhere. 
For every $m$, and any basis vector
$e_l\in\mathbb Z^k$ one can be easily find
$v\in\mathbb Z^k$ and $E_{ij}^m$ such that 
$me_i=v-E_{ij}^m(v)$. Thus it remains
to find $m$ such that $\phi(K)$ contains
every matrix $E_{ij}^m$. Let $C$ be a
finite index normal subgroup of $GL_k(\mathbb Z)$
that lies in $\phi(K)$.
Since the $m$th power of 
$E_{ij}^1$ equals $E_{ij}^m$,
it suffices to pick $m$  divisible
by the index of $C$ in $GL_k(\mathbb Z)$.

An elegant way to prove finiteness of $(\mathbb Z^k)_K$
was explained to me by Andy Putman:
One needs to show that $(\mathbb R^k)_K=0$,
and since $K$ is a lattice in $SL_k(\mathbb R)$,
the Borel density theorem shows that $\mathbb R^k$
is an irreducible $K$-module,
so the submodule generated by
$\{v-\phi_g(v)| v\in \mathbb R^k,\, g\in K\}$
cannot be proper.

Finally, if $\phi$ is injective and $n\ge 3$,
then $K$ has Kazhdan property (T), and hence $K^\ab$
is finite. 
\end{proof}

\begin{rmk}
The property of having finite abelianization
is inherited by amalgamated products and extensions:
the former is obvious and the latter follows
from right exactness of the abelianization functor.
\end{rmk}

\begin{ex}
Suppose $A$ is a normal abelian subgroup
of groups $H$ and $F$, so that
$A$ is also normal in $G:=H\ast_A F$.
Suppose $H$ and $A$ satisfy an 
algebraic condition that forces the existence
of a parabolic in $A$ for every
isometric, properly discontinuous $H$-action
on a Hadamard manifold, which happens e.g. if 
$H$ is as in Theorem~\ref{thm: into-center}, or if
$H=\mathbb Z^k\rtimes K$ as in 
Proposition~\ref{prop: Putman ex}.
Then $G$ is grounded, and the above
discussion allows one to construct many such
groups that have finite abelianization and type {\it F}.
\end{ex}

\section{Invariant Busemann functions on visibility manifolds}
\label{sec: visib}

Grounded groups are much easier to find among isometry groups
of visibility manifolds.
The main geometric ingredient is supplied by the following.

\begin{lem}
\label{lem: lem-visibi}
If $X$ is visibility, $G$ is 
not virtually cyclic,
and $X_G(\infty)$ is nonempty, then
$G$ fixes a Busemann function, $G$ contains no axial isometry,
$X_G(\infty)$ is a single point, and in particular
$G$ is grounded.  
\end{lem}
\begin{proof}
If $X_G(\infty)$ contains two distinct points,
then these points can be joined by a geodesic
as $X$ is visibility. The union of all geodesics
that join the points splits as $C\times\mathbb R$,
where $C$ is a compact convex subset. 
There exists an element of $G$ that acts as a 
nontrivial translation on the $\mathbb R$
factor (for if $G$-action on the $\mathbb R$-factor
were trivial, compactness of $C$ and proper discontinuity
of the action on $C\times\mathbb R$
would imply that finiteness of the group $G$).
Thus $G$-action on $C\times\mathbb R$ is cocompact,
so $G$ is two ended and hence virtually-$\mathbb Z$,
contradicting the assumptions.
Thus $X_G(\infty)$ is a single point, to be denoted 
$\beta$.

First we deal with the case when $G$ is periodic.
Any elliptic element $g\in G$
pointwise fixes the ray that connects
$\beta$ with any point of $X$ fixed by $g$, and hence
$g$ preserves the Busemann function for that ray, and 
hence every Busemann function for $\beta$ as they differ by a constant.

A key property of a parabolic
isometry of a visibility manifold is that
is has one fixed point at 
infinity, and similarly, any axial isometry
fixes precisely two points at 
infinity~\cite[Lemma 6.8]{BGS}.

It follows from~\cite[Proposition 4, page 8]{KarNos04}
that the fixed point sets of any two infinite order
isometries are either disjoint or equal. Since
any element of $G$ fixes $\beta$, all infinite order
isometries in $G$  have the same fixed point set.

Suppose $G$ contains an axial isometry
and no parabolic isometries.
Let $\a\in X(\infty)-\{\beta\}$ such that
each axial isometry in $G$ fixes $\a$ and $\b$.
If $e\in G$ is elliptic and $h\in G$
is axial, then $ehe^{-1}$ is axial
with fixed points $e(\beta )=\beta$ and $e(\alpha )$, so 
$e(\alpha )=\alpha $. This contradicts
that above conclusion that $X_G(\infty)=\{\beta\}$.

Finally, suppose $G$ that contains a parabolic.
Then all infinite order elements of $G$ are parabolic
with unique fixed point $\beta$. Therefore, they fix a 
Busemann function at $\beta$. By the same
argument as above, any elliptic element in $G$
also fixes the Busemann function. 
\end{proof}

In~\cite{KarNos04} Karlsson-Noskov studied isometry groups
of visibility spaces with nonempty $X_G(\infty)$.
To state their result we call
a group $G$ {\it no-$F_2$-connected\,} if
$G$ is generated by a set $S$ of infinite order
elements such that
the following graph is connected: the vertices are
elements of $S$, and two vertices $s_1, s_2\in S$ are joined by an edge
if and only if the subgroup generated by $s_1, s_2$ 
does not contain a nonabelian free subgroup.

By~\cite[Theorem 2]{KarNos04} and 
Proposition~\ref{lem: lem-visibi} if $G$ is a
no-$F_2$-connected, not virtually-cyclic,
discrete isometry group of a visibility manifold, 
then $G$ is grounded, and it 
fixes a Busemann function.  

The class of no-$F_2$-connected groups trivially
contains each torsion-free group without free nonabelian
subgroups. Also the product of any two nontrivial, 
torsion-free groups is no-$F_2$-connected
(here $S$ is the set of all elements in the product 
for which one of the coordinates equals the identity). 
More examples of no-$F_2$-connected groups
can be found in~\cite[Section 8]{KarNos04}. 

We conclude that if $X$ visibility, then Part (3) of
Theorem~\ref{thm: intro-grounded}
can be improved to ``$G$ is amenable and not
virtually cyclic'' or by ``$G$ does not contains a nonabelian
free subgroup and is generated by the set of its
infinite order elements'', and similarly, Part (5) 
can be improved to ``a normal abelian subgroup of 
$G$ is not cyclic''.

\section{Manifolds with $\mathbb R$-factors}
\label{sec: with R-factors}

In this section we discuss when an open manifold is homeomorphic
to the product of $\mathbb R$ and another manifold.
For background on ends of manifolds
we refer to~\cite{Sie-thesis, Geo-book, GuiTinII}.

Guilbault showed in~\cite{Gui-prod-R07} that if $W$ is an 
open manifold that is homotopy equivalent to a finite complex,
then $W\times\mathbb R$ has stable fundamental group at infinity
and zero Siebenmann's end obstruction, so that if $\dim(W)\ge 5$,
then $W\times\mathbb R$ is diffeomorphic to the interior of a compact manifold.
If the complex has codimension $\ge 3$ in $W\times\mathbb R$,
one can say more:

\begin{thm}\label{thm: collar reg nbhd}
Let $W$ be an open $m$-manifold that is homotopy equivalent
to a finite complex of dimension $k\le m-2$.
If $m\neq 3$, then $W\times \mathbb R$ 
is diffeomorphic to the interior of a regular neighborhood
of $k$-dimensional subcomplex.
\end{thm}

\begin{proof}
The case $m=2$ follows from the classification of open
surfaces~\cite{Ric}, which implies that $W$ 
is diffeomorphic to $\mathbb R^2$.
Suppose $m\ge 4$ so that $V:=W\times\mathbb R$
has dimension $\ge 5$. 
Stalling's theorem~\cite{Sta-emb-up-to-homot}
allows to embed any codimension $\ge 3$ finite
complex up to homotopy type, so
let $K$ be a finite
$k$-dimensional subcomplex onto 
which $V$ deformation retracts.
Let $R$ be a regular neighborhood of $K$.
The inclusions $\d R\to R\to V$ induce
$\pi_1$-isomorphisms as $K$ has codimension $\le 3$.
Excision in the universal cover as in the proof 
of~\cite[Theorem 2.1]{Sie-collar}
implies that $U:=V\setminus\mathrm{Int}(R)$
deformation retracts onto $\d U=\d R$. 

The manifolds $V$ and $W$  are one-ended as 
their homology groups (with twisted coefficients)
vanish in degrees 
$\ge m-2$~\cite[Proposition 1.2]{Sie-collar}.
Denote their ends by $\e_{_V}$, $\e_{_W}$.
By~\cite[Propositions 3.1, 3.2]{Gui-prod-R07}
the end $\e_{_V}$ has arbitrary small
finitely dominated $0$-neighborhoods and its 
fundamental group is stable. Let $\pi_1(\infty_{_V})$ 
denote the inverse limit of the fundamental
groups of the $0$-neighborhoods; this group is
finitely presented being a retract of the (finitely presented)
fundamental group of a $0$-neighborhood.

As $R$ is compact, the above discussion of 
$\e_{_V}$ applies to the end $\e_{_U}$ of $U$, and in particular,
$\pi_1(\infty_{_U})$ and $\pi_1(\infty_{_V})$ are isomorphic
via the inclusion.
If we show that 
the canonical map $e_{_U}\co\pi_1(\infty_{_U})\to \pi_1(U)$
is an isomorphism, then
Siebenmann's open collar theorem~\cite[Theorem 1.6]{Sie-collar}
would imply that $U$ is diffeomorphic to $\d U\times [0,1)$,
which would prove the theorem. 
Since the inclusion $U\to V$
induces a $\pi_1$-isomorphism, $e_{_U}$ is essentially the canonical map 
$e_{_V}\co\pi_1(\infty_{_V})\to \pi_1(V)$.

Next we note that $e_{_U}$ is onto and has perfect kernel.  
To see that multiply $U$ by a circle to make the 
dimension $\ge 6$. 
By Siebenmann's product formula for the end 
obstruction~\cite[Theorem 7.5]{Sie-thesis}, 
$S^1\times U$ is the interior 
of a compact manifold, which is a one-sided $h$-cobordism
that deformation retracts onto the
boundary component of the form $S^1\times\d U$.
The other boundary component $B$ satisfies
$\pi_1(B)\cong\pi_1(S^1)\times\pi_1(\infty_{_U})$.
Restricting the retraction to $B$ gives
a map $B\to S^1\times\d U$, which is 
a $\pi_1$-surjection with perfect kernel
by a standard argument involving
Poincar\'e duality in the universal cover as e.g.
in~\cite[Theorem 2.5]{GuiTinII}.
Under the obvious identifications this
$\pi_1$-surjection corresponds
to the product of $e_{_U}$
and the identity of $\pi_1(S^1)$, so
$e_{_U}$ is surjective with perfect kernel.

On the other hand,
it follows from~\cite[Propositions 3.5]{Gui-prod-R07}
that $\pi_1(\infty_{_V})$ is isomorphic to
a (possibly trivial) amalgamated product
of the form $\pi_1(W)\ast_{_\G} \pi_1(W)$, and moreover,
under the isomorphism $e_{_V}$ is identified with
the standard projection 
$\pi_1(W)\ast_{_\G} \pi_1(W)\to \pi_1(W)$ that extends 
the identity maps on each factor. The kernel of this
is a free group. (Indeed, the kernel is a normal subgroup that
intersects the factors trivially, and hence intersects
trivially each of their conjugates. Thus
the kernel acts freely on the Bass-Serre tree for the amalgamated 
product, and hence it is a free group).
Since a free perfect group must be trivial, 
$e_{_U}$ is an isomorphism, which finishes the proof. 
\end{proof}

\begin{rmk} 
\label{rmk: m=3 case collar}
If $m=3$ and $\pi_1(W)$ is trivial or $\mathbb Z$, which
are the fundamental groups of $m-2$-complexes for which the 
topological surgery is known to work, then the same proof,
based on the open collar theorem of~\cite{Gui-collar92}, shows
that $W\times\mathbb R$ is homeomorphic to the interior of a regular
neighborhood of $K$, i.e. $D^4$, or $D^3\times S^1$, respectively. 
\end{rmk}

\begin{rmk}
\label{rmk: prod of contr mnflds}
The product of $\mathbb R$ with
any aspherical $(n-1)$-manifold  
is covered by $\mathbb R^{n}$. In fact,
the references below imply that
the product of any two open contractible manifolds
is diffeomorphic to $\mathbb R^n$
(provided their dimensions are positive and add up to $n$). 
This is due to Stallings~\cite{Sta62} if $n\ge 5$.
The case of the product of an open contractible
$3$-manifold with $\mathbb R$ 
is due to Luft~\cite[Theorem 5]{Luf87}
and Perelman's solution of Poincar\'e conjecture.
Finally, any contractible
open $2$-manifold is diffeomorphic to $\mathbb R^2$
by the classification of surfaces~\cite{Ric}.
These results only give a PL homeomorphism
onto $\mathbb R^n$, but a smooth
structure on $\mathbb R^n$ is uniquely determined by the 
PL structure it induces~\cite[Corollary 6.6]{Mun60}.
\end{rmk}

\begin{cor}
\label{cor: product aspherical Rn}
Let $C$ be an open aspherical $l$-manifold and 
$B$ be a closed manifold such that $B\times C$ is homeomorphic to the product
of $\mathbb R$ and an open manifold. Then the universal 
cover of $C$ is homeomorphic to $\mathbb R^l$.
\end{cor}

\begin{proof} We can assume $\dim(C)\ge 3$, and 
that $C$ is contractible, because the property of being
homeomorphic to the product of $\mathbb R$ and an open manifold
in inherited by any cover. Let $W$ be an open manifold such that 
$B\times C$ is homeomorphic to $W\times\mathbb R$.
By~\cite[Propositions 3.1-3.2]{Gui-prod-R07} 
$W\times\mathbb R$, and hence $B\times C$,
has only one end at which $\pi_1$ is stable, and
$\pi_1(\infty)$ is finitely presented. Hence the same holds
for the end of $C$. It follows that 
$\pi_1(\infty)$ is isomorphic to $\pi_1(B)\times\pi_1(\infty_{_C})$
and the projection on the first factor can be identified
with the canonical map $\pi_1(\infty)\to \pi_1(B\times C)$,
which is an isomorphism by Theorem~\ref{thm: collar reg nbhd}
and Remark~\ref{rmk: m=3 case collar}.
Thus $\pi_1(\infty_{_C})$ is trivial. 
Any open contractible $l$-manifold
that is simply-connected at infinity is homeomorphic to $\mathbb R^l$.
(This is due to~\cite{Sta62} if $l\ge 5$ and 
to~\cite{Gui-collar92} if $l=4$, while the case $l=3$ follows 
from~\cite{HusPri} and the non-existence of fake $3$-cells
by Perelman's solution of the Poincar\'e conjecture).
\end{proof}

\begin{cor} Let $F\to E\to B$ be a fiber bundle 
whose base $B$ is a closed manifold with $\pi_2(B)=0$, 
and fiber $F$ is a compact
contractible $l$-manifold. 
If $\mathrm{Int}(E)$ is homeomorphic
to $W\times\mathbb R$ for a manifold $W$, then
$\mathrm{Int}(F)$ is homeomorphic to $\mathbb R^l$.
\end{cor}
\begin{proof} 
We can assume $l>2$; in particular if $\dim(W)=3$,
then $\dim(B)\le 1$. Theorem~\ref{thm: collar reg nbhd} and 
Remark~\ref{rmk: m=3 case collar} imply that $\mathrm{Int}(E)$
is homeomorphic to the interior of a regular neighborhood
of a subcomplex of codimension $>2$, hence the inclusion
$\d E\to E$ induces a $\pi_1$-isomorphism.
The same is then true for the the bundle projection $\d E\to B$,
whose fiber is $\d F$. Since $\pi_2(B)=0$,
the homotopy sequence of the fibration $\d F\to \d E\to B$
implies $\pi_1(\d F)=0$, so $\Int(F)$ is simply-connected at infinity, and
hence homeomorphic to $\mathbb R^l$, 
as in the proof of Corollary~\ref{cor: product aspherical Rn}.
\end{proof}

\begin{proof}[Proof of Corollary~\ref{cor-intro: Rn}]
If $W$ is homeomorphic to $\mathbb R^n$, then
it is diffeomorphic to $\mathbb R^n$ when $n\neq 4$,
and $W\times S^1$ is diffeomorphic to 
$\mathbb R^4\times S^1$~\cite[Theorem 2.2]{Sie-collar}.
Of course, $\mathbb R^n\times S^1$ admits a metric of $K\le 0$.
Conversely, suppose $W\times S^1$
admits a metric of $K\le 0$, and write it as $X/G$
where $G$ is infinite cyclic.
The generator of $G$
is either axial or parabolic.
In the former case $X/G$ is a vector bundle
over $S^1$ as the normal exponential map to
a totally geodesic submanifold of $X$ is a diffeomorphism,
and the bundle is trivial because $W\times S^1$ is orientable.
In the latter case $G$ fixes a Busemann function,
so $X/G$ is the product of $\mathbb R$ and an open manifold,
and the claim follows from 
Corollary~\ref{cor: product aspherical Rn}.
\end{proof}

\section{Manifolds without $\mathbb R$-factors}
\label{sec: without R-factors}

The purpose of this section is to give a method
of constructing $\mathbb R^n$-covered manifold
that are not homeomorphic to 
$\mathbb R^{n-1}\times\mathbb R$
covered manifolds.

\begin{thm}
\label{thm: count many not Rn-1xR}
Let $V$ be an open manifold that is
covered by $\mathbb R^n$, $n\ge 5$,
and is homotopy equivalent to a finite complex
of dimension $\le n-3$.
If $\pi_1(V)$ is nontrivial, then the tangential homotopy type of $V$ contains 
countably many open manifolds that are covered by $\mathbb R^n$ 
but are not homeomorphic to manifolds covered by 
$\mathbb R\times\mathbb R^{n-1}$.
\end{thm}
\begin{proof}
Stalling's theorem~\cite{Sta-emb-up-to-homot}
allows one to embed any codimension $\ge 3$ finite
complex up to homotopy type, so
let $N\subset V$ be a regular neighborhood of 
that subcomplex.
Since the subcomplex 
has codimension $\le 3$, we know that
$\d N$ is connected and
the inclusions $\d N\to N\to V$ induce
$\pi_1$-isomorphisms. The universal cover
of $\Int(N)$ is diffeomorphic to $\mathbb R^n$
because it is a regular neighborhood of a codimension
$\ge 3$ subcomplex, which easily implies
that it is simply-connected at infinity,
as mentioned e.g. 
in the statement of~\cite[Theorem 2.1]{Sie-collar}.

Thus there is an embedded circle in $\d N$ that
is homotopically nontrivial in $N$. By replacing
it with a embedded circle homotopic to the square
of the original one, we may assume its normal bundle
is trivial (and it is still homotopically nontrivial
in $\pi_1(N)$ since this group is torsion-free). 
Denote the circle by $c$, and
its closed tubular neighborhood by $T_c$.

Let $C$ be any compact contractible $(n-1)$-manifold 
bounded by a non-simply-connected homology sphere
whose fundamental group is freely indecomposable
(see e.g. Appendix~\ref{sec: homology sph}, 
or~\cite{CurKwu65, Gla67}).
Fix a closed coordinate disk $\Delta\subset \d C$,
attach $C\times S^1$ to $N$ by identifying
$\Delta\times S^1$ with $T_c$, and denote
the resulting compact manifold with boundary by
$N_{c,C}$
(the choices involved in this identification will
be irrelevant for our purposes).

The fundamental group of the boundary of $N_{c,C}$  
is the amalgamated product of $\pi_1(\d N)$ and
$\pi_1(\d C)\times \mathbb Z$ along the infinite
cyclic subgroup generated by the homotopy class of $c$.
Taking the quotient by the normal closure
of $\pi_1(\d N)$ we see that the amalgamated product 
surjects onto $\pi_1(\d C)$,
so its rank is at least as large as the rank of $\pi_1(\d C)$.

By taking boundary connected sums of $C$
with itself one gets a sequence $\{C_i\}_{i\in\mathbb N}$ of compact contractible
manifolds such that $C_i$ bounds the connected sum $\Sigma_i$ of 
$i$ copies of the homology sphere $\Sigma:=\d C$. 
Grushko's theorem implies $\mathrm{rank}(\pi_1(\Sigma_i))=
i\,\mathrm{rank}(\pi_1(\Sigma))$, so
$V_i:=\Int(N\#_{c,C_i})$ are pairwise 
non-homeomorphic because their fundamental groups
of infinity are pairwise non-isomorphic. 
Deformation retraction $C_i\to \Delta$ extends
to a deformation retraction of $N_{c,C_i}\to N$,
so $V_i$ is tangentially homotopy equivalent to $V$.

Next we show that $V_i$ is covered by $\mathbb R^n$.
Denote the universal cover of $N$ and $V$
by $\widetilde N$ and $\widetilde V$, respectively.
Since $c$ is homotopically nontrivial,
it unravels in $\widetilde N$ to a 
proper embedding $\mathbb R\to\d\widetilde N$.
Let $T$ be the union of all lifts of $\Int(T_c)$ to
$\d\widetilde N$. Each component of $T$ is an open
disk, and these disks have disjoint closures.
Then $\widetilde V\cup T$ is a noncompact manifold
with boundary $T$. Note that $\widetilde V\cup T$ 
can be compactified to $D^n$ where $T$ becomes
the union of a countable collection of 
round open disks with pairwise
disjoint closures, namely, attach 
a collar to $T$ and invoke the strong version of
Cantrell-Stalling's hyperplane linearization 
theorem~\cite[Corollary 9.3]{CKS}.
On the other hand, 
the universal cover of $S^1\times (\Int(\Delta)\cup\Int(C_i))$
is a noncompact manifold with boundary whose
interior is diffeomorphic to $\mathbb R^n$~\cite{Sta62}
and whose boundary is an open disk, so it also
compactifies to $D^n$ so that the disk 
$\mathbb R\times\Int(\Delta)$ becomes
a round disk. It follows that attaching
each lift of $S^1\times (\Int(\Delta)\cup\Int(C_i))$
to $\tilde V\cup T$ simply amounts to attaching
an open collar to each component of $T$, and hence
the result has interior diffeomorphic to $\widetilde V$,
which is $\mathbb R^n$.

It remains to show that $V_i$ is not homeomorphic
to a manifold covered by $\mathbb R\times\mathbb R^{n-1}$. 
If it is, then Theorem~\ref{thm: collar reg nbhd} applies
and the map $\pi_1(\infty)\to\pi_1(V_i)$
is an isomorphism. Hence the same is true
for the inclusion $\d N_{c,C_i}\to N_{c,C_i}$, which on the other hand
is not injective because any homotopically
nontrivial loop in $\Sigma_i$
is null-homotopic in $C_i$. This contradiction completes the proof. 
\end{proof}

We can replace the conclusion ``countably many'' in 
Theorem~\ref{thm: count many not Rn-1xR}
by "a continuum of'' under the following technical assumption. 
Let $\mathcal S$ be the class of groups 
of the form $\pi_1(\Sigma^3\times S^1)$
where $\Sigma^3$ is an aspherical homology $3$-sphere
that (smoothly) bounds a contractible manifold. Let us
fix a set $\mathcal S_0$ of pairwise non-isomorphic groups 
in $\mathcal S$ such that any group in $\mathcal S$ is isomorphic to a group in 
$\mathcal S_0$; the set 
$\mathcal S_0$ is countably infinite (see Appendix~\ref{sec: homology sph}), hence we write 
$\mathcal S_0=\{\pi_1(\Sigma_i^3\times S^1)\}_{i\in\mathbb N}$.
Given an infinite subset $\mathfrak S$ of $\mathcal S_0$,
we say that a group $G$ is {\it no-$\mathfrak S$}
if {\bf one} of the following conditions (1), (2) holds: 
\vspace{-4pt}
\begin{itemize}
\item[\textup{(1)}] 
no group in $\mathfrak S$ can be embedded 
into $G$.
\item[\textup{(2)}]\vspace{2pt}
$G$ splits as a (possibly infinite and possibly consisting
of a single vertex) graph of groups such that
\begin{itemize}
\item[\textup{(a)}]\vspace{2pt}
all edge groups have cohomological 
dimension $\le 2$,
\item[\textup{(b)}]\vspace{2pt} 
no vertex group is isomorphic
to a group in $\mathfrak S$. 
\item[\textup{(c)}]\vspace{2pt}
no vertex group
splits over a trivial or an infinite cyclic subgroup.
\end{itemize}
\end{itemize}

Here {\it splits} stands
for {\it splits as an amalgamated product or an
HNN-extension}. 
Some examples of no-$\mathfrak S$ groups are given below:
\begin{enumerate} 
\item[(i)] If $G$ has cohomological dimension $<4$, then $G$ satisfies (1) because
groups in $\mathfrak S$ has cohomological dimension $4$.
\item[(ii)] If $G$ does not splits over a trivial or an infinite cyclic subgroup,
and $G\notin\mathfrak S$,
then $G$ satisfies (2) for the graph of groups that 
consists of a single vertex.
\item[(iii)] 
Suppose $\mathfrak S$ is associated with 
the infinite family of $\Sigma_i$'s that are modelled on 
$\widetilde{SL}_2(\mathbb R)$ geometry and 
described in Appendix~\ref{sec: homology sph}.
Then $\mathbb Z^3$ and $\mathbb Z\ast\mathbb Z$
embed into every group in $\mathfrak S$.
Thus if $\mathbb Z^3$ or $\mathbb Z\ast\mathbb Z$
does not embed into $G$,
then $G$ is no-$\mathfrak S$ by (1). 
\item[(iv)]
If $G$ splits as a finite graph of groups that satisfies
(a) and (c), then it can be arranged to satisfy (b) 
by removing the (finitely many) vertex groups from
$\mathfrak S$. 
\end{enumerate} 

\begin{prop}
Let $G$ be an $n$-dimensional Poincar{\'e} duality group
(e.g. the fundamental group of a closed aspherical $n$-manifold). Then $G$
is no-$\mathfrak S$ if and only if $G\notin\mathfrak S$.
In particular, $G$ is no-$\mathfrak S$ for some $\mathfrak S$.
\end{prop}
\begin{proof}
The case $n<4$ is trivial: $G$ is no-$\mathfrak S$ by (1) and 
$G\notin\mathfrak S$. In the case $n\ge 4$ a key ingredient is that
any splitting of $G$ as an amalgamated product $G_1\ast_A G_2$ over a subgroup $A$ of cohomological 
dimension $\le n-2$ is trivial, i.e. $G$ equals a vertex subgroup
(see e.g.~\cite{Bel-intersec} for a proof in the case of closed aspherical manifolds
which extends to Poincar\'e duality groups). 
Thus if $G\notin\mathfrak S$, then $G$ satisfies (2) for the graph of group consisting 
of a single vertex. Conversely, arguing by contradiction suppose that $G$ is no-$\mathfrak S$
and $G\in\mathfrak S$. Then (1) is not true, and  
since $G\cong\pi_1(\Sigma_i^3\times S^1)$ 
admits no nontrivial splitting over a subgroup of dimension $\le 2$, 
the only possible graph in (2) in a single vertex which contradicts (b).
\end{proof}

\begin{prop}
If $G$ is the free product of all groups in $\mathcal S_0$, then $G$
has finite cohomological dimension, yet there is no $\mathfrak S$ for which $G$ is no-$\mathfrak S$.
\end{prop}
\begin{proof} Clearly $G$ is the fundamental group
of a locally finite $4$-dimensional complex, so it has finite cohomological dimension.
Write $\mathcal S=\{H_i\}$, assume arguing by contradiction that
$G$ is no-$\mathfrak S$, and fix a group $H_j\in\mathfrak S$. 
Clearly (1) is not true, 
so $G$ has a graph of group decomposition coming from (2). 
Let us compare that graph of groups decomposition with the factorization $G={\ast}_i H_i$.
Since no $H_i$ splits over subgroups of dimension $\le 2$,
Bass-Serre theory considerations show that 
$H_j$ lies in a conjugate $K$ of a vertex group in the graph of group, and similarly
(c) implies that $K$ lies in a conjugate of some $H_l\in \mathcal S_0$. Now $K\neq H_j$ by (b). Hence $H_j$
fixes two vertices in the Bass-Serre tree for $G=\ast_i H_i$, which contradicts nontriviality of $H_j$.
\end{proof}

\begin{quest}
Is it true that every group of type {\it F} is no-$\mathfrak S$ for some 
$\mathfrak S$?
\end{quest}

\begin{thm}
\label{thm: not reg nbhs}
Let $V$ be an open aspherical $n$-manifold that is
homotopy equivalent to a finite complex
of dimension $\le n-3$, where $n\ge 5$ 
and $\pi_1(V)$ is no-$\mathfrak S$ for some
$\mathfrak S$.
Then the tangential homotopy type of
$V$ contains a continuum of
manifolds covered by $\mathbb R^n$, 
and not homeomorphic to manifolds covered by 
$\mathbb R\times\mathbb R^{n-1}$.
\end{thm}
\begin{proof}
As in the proof of Theorem~\ref{thm: count many not Rn-1xR}
we can change $V$ within its tangential homotopy type to assume that
$V$ is the interior of a regular neighborhood $N$
of a codimension $\ge 3$ finite subcomplex.
Set $B:=\d N$, so that $\pi_1(B)\to \pi_1(N)\cong\pi_1(V)$
is an isomorphism induced by the inclusion.

Fix $\mathfrak S=\{\pi_1(\Sigma_i^3\times S^1)\}$ such that
$\pi_1(V)$ is no-$\mathfrak S$.
Results of Kervaire, recalled 
in Appendix~\ref{sec: homology sph} 
ensure the existence of an infinite sequence of
compact contractible $(n-1)$-manifolds 
$\{C_i\}_{i\in\mathbb N}$ 
such that $\pi_1(\d C_i)\cong
\pi_1(\Sigma_i^3)$. 
Set $\Sigma_i:=\d C_{i}$. 

Pick two disjoint closed coordinate disks 
$\Delta_{i-}$, ${\Delta}_{i+}$ in $\Sigma_i$.
Given a subset $\a\subseteq\mathbb N$, let $C_\a$ be
a boundary connected sum of $C_i$'s
with indices in $\a$ defined as follows:
if $\a=\{i_1,\dots, i_k,\dots\}$
with the ordering induced from $\mathbb N$, then
we attach $C_{i_k}$ to $C_{i_{k+1}}$ by identifying 
${\Delta}_{i_k+}$ and $\Delta_{i_{k+1}-}$ for all $k$.

Pick a nontrivial element in $\pi_1(B)$.
By replacing the element with its square, which we denote $s$, 
we can assume that it can be represented by 
an embedded circle in
$B$ with trivial normal bundle.
Attach $N$ to $C_\a\times S^1$ 
by identifying $\Delta_{i_1-}\times S^1$
with a closed tubular neighborhood
of the above embedded circle in $B$ via an arbitrary
homeomorphism, denote the result by $N_\a$,
and set $V_\a:=\Int(N_\a)$.
Since $N_\a$ deformation retracts to $N$, 
the manifolds $V_\a$ and $V$ are tangentially
homotopy equivalent.
The corresponding part of the
proof of Theorem~\ref{thm: count many not Rn-1xR}
implies that $V_\a$ is covered by $\mathbb R^n$ 
but is not homeomorphic to a manifold covered by 
$\mathbb R\times\mathbb R^{n-1}$.
It remains to show that $\{V_\a\}$ fall into a continuum of
homeomorphism types.

We next construct a nested
sequence $\{W_k\}$ of codimension zero
submanifolds that are neighborhoods of infinity 
in $\Int(C_\a)$
such that $\Int(W_k)\supset W_{k+1}$, the boundary 
$\d W_k$ is compact, 
and $\bigcap_{k\ge 1} W_k=\emptyset$.
(Any sequence of neighborhoods that satisfy the 
last property is called {\it cofinal}\,).
Imagine a worm that eats the inside
of $C_{i_1}$ except for a collar neighborhood
of $\Sigma_{i_1}$, which gives $W_1$. 
For the second helping the worm eats some more of
$\Int(C_{i_1})$ leaving a smaller
collar neighborhood of $\Sigma_{i_1}$,
and then eats a tunnel into $C_{i_2}$ through $\Delta_{i_2-}$,
by leaving a collar neighborhood of $\d \Delta_{i_2-}$
in $\Delta_{i_2}-$, 
and then eats the inside of $C_{i_2}$ except
for a collar neighborhood
of $\Sigma_{i_2}$, which results in $W_2$. 
The process is repeated so that the worm gets into $C_{i_k}$
on the step $k$, which yields $W_k$. Then
$\pi_1(W_1)\cong\pi_1(\Sigma_{i_1})$, and in general
$\pi_1(W_k)\cong \pi_1(\Sigma_{i_1})
\ast\dots\ast\pi_1(\Sigma_{i_k})$.

Now $\{W_k\times S^1\}$ is a sequence of neighborhoods
of infinity in $\Int(C_\a)\times S^1$.
Let $\{O_k\}$ be a nested cofinal sequence of 
neighborhoods of infinity of $\Int(N)$ such that
$O_k\bigcup B$ is a collar neighborhood of $B$.
Let $\{X_k\}$ be a nested cofinal sequence of closed
collar neighborhoods of $\d \Delta_{i_1-}$ in $\Delta_{i_1-}$,
and let $Z_k$ be the complement of $X_k$ in $\Delta_{i_1-}$.
We let the worm enlarge $O_k\cup (W_k\times S^1)$ 
by eating through the collar neighborhoods of $Z_k\times S^1$
in $N$ and $C_a\times S^1$ 
until $O_k$ connects to $W_k\times S^1$ through a tunnel homeomorphic to
$[-\frac{1}{k}, \frac{1}{k}]\times Z_k\times S^1$; we denote
the result by $U_k$. 
Thus $\{U_k\}$ is a nested, cofinal sequence of
neighborhood of infinity $\{U_k\}_{k\ge 1}$ in 
$V_\a=\Int(N_\a)$
such that $\Int(U_k)\supset U_{k+1}$ and the boundary 
$\d U_k$ is compact. By van Kampen's theorem $\pi_1(U_k)$
is the amalgamated product of $\pi_1(W_k\times S^1)$ 
and $\pi_1(B)$
along the infinite cyclic subgroup generated by $s$.

Given basepoints $p_k\in U_k$ and paths $r_k$
in $U_{k}$ connecting $p_k$ to $p_{k+1}$, 
we get an inverse sequence of homomorphisms
$\{\rho_k\co\pi_1(U_k, p_k)\longleftarrow 
\pi_1(U_{k+1}, p_{k+1})\}$
induced by the inclusion 
$U_{k+1}\subset U_k$ followed by the basepoint change 
conjugation via $r_k$'s. 
Each $\rho_k$ 
has a section defined as the composition of 
homomorphisms induced by the following operations:
deformatation retract 
$\bigcup_{m>k}C_{i_m}\times S^1$ 
onto ${\Delta}_{i_k+}\times S^1$
inside $U_k$, intersect the result with $U_{k+1}$, 
include the intersection into $U_{k+1}$, and change
the base point via $r_k$. Except for 
including into $U_{k+1}$, these operations induce
fundamental group isomorphisms.

Two inverse sequence 
of homomorphisms $\{A_l\longleftarrow A_{l+1}\}$,
$\{B_m\longleftarrow B_{m+1}\}$
are {\it pro-equivalent\,}
if after passing to subsequences
there is a commutative diagram
\begin{equation*}
\xymatrix{
A_{l_1}&
&A_{l_2} \ar[ll] \ar[ld]&
&A_{l_3}\ar[ll] \ar[ld]&
&\dots\ar[ll] \ar[ld]
\\
&B_{m_1}\ar[lu]&
&B_{m_2}\ar[lu]\ar[ll]&
&B_{m_3}\ar[lu]\ar[ll]&
&\dots\ar[ll]
}
\end{equation*}
in which the horizontal arrows are the
composites of the given homomorphisms
(see e.g.~\cite[2.1]{GeoGui} for details).
A sequence of homomorphisms is {\it semistable\,}
if it is pro-equivalent to a sequence of surjective
homomorphisms. 

Since $\rho_k$ has a section, it is surjective, so
the sequence $\{\rho_k\}$ is semistable, which implies 
(see e.g.~\cite[2.2]{GeoGui}) 
that its pro-equivalence class is independent of
any of the choices of $U_k$, $p_k$, $r_k$,
and hence it depends only on the homeomorphism type
of $\Int(N_\a)$. If we set $G_k^\a:=\pi_1(U_k, p_k)$,
and identify $G_k^\a$ with a subgroup
of $G_{k+1}^\a$ via the above section, then $\rho_k$
becomes a retraction $G_k^\a\longleftarrow G_{k+1}^\a$. 

For the rest of the proof we think of
$G_k^\a$ as an amalgamated product of
subgroups isomorphic to 
$\pi_1(B), \pi_1(\Sigma_{i_1}\times S^1),\dots,
\pi_1(\Sigma_{i_k}\times S^1)$ along $\langle s\rangle$
and refer to these subgroups as {\it factors}.
Note that $G_k^\a$ is the subgroup of $G_{k+1}^\a$
generated by the above factors, and the kernel
of the retraction $G_k^\a\longleftarrow G_{k+1}^\a$
equals the normal closure of the subgroup 
$\pi_1(\Sigma_{i_{k+1}})\times \{1\}$ of the
$\pi_1(\Sigma_{i_{k+1}})\times\mathbb Z$-factor, 
which is mapped to $G_k^\a$ by projection
onto $\mathbb Z$, identified with $\langle s\rangle$.

To get a continuum of homeomorphism types of $\Int(N_\a)$'s
it suffices to show that that the pro-equivalence of 
$\{G_k^\b\longleftarrow G_{k+1}^\b\}$ and 
$\{G_k^\g\longleftarrow G_{k+1}^\g\}$ implies equality
of the subsets $\b$, $\g$, and by symmetry we only need
to show that any $i_k\in\b$ also lies in $\g$.
Look at a portion
of the commutative diagram from the definition of 
pro-equivalence:
\begin{equation*}
\xymatrix{
G_{l_1}^\b&
&G_{l_2}^\b \ar[ll]_{q_1} \ar[ld]&
&G_{l_3}^\b\ar[ll]_{q_2} \ar[ld]^{\psi}&
\\
&G_{m_1}^\g\ar[lu]&
&G_{m_2}^\g\ar[lu]^{\phi}\ar[ll]_p&
}
\end{equation*}
with $l_2>l_1>i_k$ and $m_2>m_1>i_k$.
With the above identifications 
$q_1$ and $q_2$ both restrict to the identity on $G_k^\b$.
Commutativity implies that $p\circ\psi$ and $\psi$
restricted to $G_k^\b$ are injective, and hence
so is $\phi$ restricted to $\psi(G_k^\b)$.
 
Let $Q\le G_{l_3}^\b$ denote the subgroup
corresponding to
the $\pi_1(\Sigma_{i_k}\times S^1)$-factor.
Note that $Q$ does not split over a trivial or an infinite cyclic subgroup because 
$\Sigma^3_{i_k}\times S^1$ is a closed aspherical
$4$-manifold, and the fundamental group
of a closed aspherical manifold of dimension $>2$
cannot have such splittings.

Since $G_{m_2}^\g$ is an amalgamated product over 
$\langle s\rangle$,
the group $\psi(Q)$ lies in a conjugate, denoted  $R$, 
of one of the factors of $G_{m_2}^\g$. 
The map $p$ restricted to a conjugate of a factor
is either injective or is a projection onto $\langle s\rangle$.
The latter is impossible for $p\vert_R$ since the image contains
$p(\psi(Q))\cong Q$. It follows that $\phi$
restricted to $R$ is injective.
Now $R$ is  isomorphic either to $\pi_1(B)$ or 
to a group in $\mathfrak S$. 
Let us consider the former
case, and use that $\pi_1(B)\cong\pi_1(V)$ is no-$\mathfrak S$.
As $\psi(Q)\le R$, the assumption (1) gives a contradiction.
The assumption (2) implies that $\psi(Q)$ lies
in a vertex group $R_0$ in a graph of groups decomposition 
of $\pi_1(N)$ as in (2) such that $R_0$
does not split over
a trivial or an infinite cyclic subgroup.

Then $\phi(R_0)$ lies in a
conjugate, denoted $S$, of one of the factors of $G_{l_2}^\b$, 
which implies $S\supset \phi(\psi(Q))=Q$. 
It follows that $S= Q$ for otherwise $S$ fixes a vertex
(in the Bass-Serre tree of the amalagamated
product decomposition of $G_l^\b$) that is 
distinct from the vertex fixed by $Q$, so
$Q$ would fix the path joining
the two vertices, but edge stabilizers are cyclic 
and $Q$ is not.
Thus $Q=\phi(\psi(Q))\subseteq \phi(R_0)\subseteq S=Q$,
so injectivity of $\phi$ on $R_0$ implies $R_0=\phi (Q)$,
which contradicts (b). 

Finally, if $R$ is isomorphic to a group in 
$\mathfrak S$, then $R$ also does not split over
a trivial or an infinite cyclic subgroup,
so the argument of the previous paragraph
gives $R=\psi(Q)$, which 
means that a group in $\mathfrak S$
is among the factors of $G^\g_{m_2}$
other than the $\pi_1(B)$-factor, i.e.
$i_k$ lies in $\g$.
\end{proof}

The same methods yield the following two results that give other
examples of manifolds that are covered by
$\mathbb R^n$ and admit no complete metric of $K\le 0$.

\begin{thm}
\label{thm: comp countable}
Let $N$ be a compact $n$-manifold 
with boundary such that $\pi_1(N)$ is nontrivial,
and $V:=\mathrm{Int}(N)$ is covered
by $\mathbb R^n$, $n\ge 5$. Then the tangential homotopy type of $N$ contains
countably many compact manifolds $N_i$ such that $\mathrm{Int}(N_i)$ 
is covered by $\mathbb R^n$ and a component of $\d N_i$ is not
$\pi_1$-injectively embedded.
\end{thm}
\begin{proof} By assumptions $\pi_1(N)$ is nontrivial, and hence infinite.
Fix a component $B$ of $\d N$. The inclusion $B\hookrightarrow N$ cannot 
be trivial on $\pi_1$. (Otherwise, passing to the universal cover we conclude
that $B$ is homeomorphic to the boundary of a noncompact contractible manifold $W$.
Noncompactness of $W$ implies it admits no constant
function with compact support, so $H^0_c(W)$ vanishes but
by the Poincar\'e duality $H^0_c(W)\cong H_{n}(W,B)$ and the
latter group is isomorphic to $H_{n-1}(B)$ by the long exact sequence of the
pair. Vanishing of $H_{n-1}(B)$ contradicts compactness of $B$).
So there is an embedded curve $c\subset B$ with trivial normal
bundle that is homotopically nontrivial in $N$,
and the proof proceeds as in Theorems~\ref{thm: count many not Rn-1xR}. 
\end{proof}

\begin{thm}
\label{thm: uncount compact}
Let $N$ be a compact manifold 
with boundary such that $\pi_1(N)$ is nontrivial, 
$V:=\mathrm{Int}(N)$ is covered
by $\mathbb R^n$ with $n\ge 5$, and there is a component $B$ of $\d N$
such that $\pi_1(B)$ is no-$\mathfrak S$ for some
$\mathfrak S$.
Then the tangential homotopy type of
$V$ contains a continuum of
manifolds covered by $\mathbb R^n$ none of
which is homeomorphic to the interior of a compact manifold.
\end{thm}
\begin{proof}
Find a curve $c$ as in the proof of Theorem~\ref{thm: comp countable}.
The proof proceeds as in Theorems~\ref{thm: not reg nbhs}
except that $O_k\subset V$ is the sequence of the neighborhoods 
of the end that corresponds to $B$. If $\a$ is infinite, then
the sequence of groups $G_k^\a$ does not stabilize, so $O_k$
is not the interior of a compact manifold.  
\end{proof}

\begin{ex}
Let $M$ be closed aspherical $(n-1)$-manifold such that 
any properly discontinuous, isometric $\pi_1(M)$-action
on a Hadamard $n$-manifold fixes a Busemann function.
Let $V$ be  complete $n$-manifold of $K\le 0$ that is homotopy equivalent to $M$.
Then $V$ is the total space of a real line bundle over $M$, and in particular,
$V$ is the interior of a compact manifold $N$ such that 
no homotopically nontrivial loop in $\d N$ is null-homotopic in $N$. 
Now Theorems~\ref{thm: comp countable} and~\ref{thm: uncount compact} 
applied to this $N$ give infinitely many manifolds
in the tangential homotopy type of $V$ that are covered by $\mathbb R^n$ and
admit no complete metric of $K\le 0$. 
This argument applies when
\begin{itemize}
\item $M$ is an infranilmanifold (see Example~\ref{intro-ex: flat} and 
Proposition~\ref{prop: nil not virt split});
\item
a finite cover of $M$ is the total space 
of a circle bundle with nonzero rational Euler class (see 
Proposition~\ref{prop: euler class});
\item
$\pi_1(M)=G$ is a higher rank, irreducible lattice as 
in Example~\ref{intro-ex: higher rank};
\item 
$\pi_1(M)=G$ a polycyclic group as in Example~\ref{ex: dekimpe poly}.
\end{itemize}
\end{ex}

\appendix

\section{Finiteness properties for groups}
\label{sec: app on finitenss}

A group $G$ {\it has type {\it F}\,} 
if there is a finite CW-complex that is  $K(G,1)$.
Our main results apply to groups of type {\it F}, 
or more generally, groups of finite cohomological dimension. 
This appendix contains references justifying
that various groups constructed in
Sections~\ref{sec: center}-\ref{sec: grounded}
have these properties (provided they are torsion-free).
We refer to~\cite{Bie-book, Bro-book}
for properties of cohomological 
and geometric dimensions of a group, of which we
mostly use the former, denoted $\cd$.
Finiteness conditions on groups, such as
{\it F} and {\it FP}, are reviewed 
in~\cite[Appendix F]{Dav-book08}. 
The following facts are known:

\begin{enumerate}
\item
Finiteness of cohomological dimension is inherited by
extensions and amalgamated 
products~\cite[Proposition VIII.2.4]{Bro-book}.
\item
If $G_1$, $G_2$, $A$ have type $F$, then so does $G_1\ast_A G_2$, cf.~\cite[Theorem II.7.3]{Bro-book}.
\item
The property of having type {\it F} or
{\it FP} is inherited by
extensions. 
(Combine~\cite[Exercise VIII.6.8]{Bro-book}
with the fact that finite presentability of a group
inherited by extensions~\cite[2.4.4]{Rob-grpth-book})
\item
The fundamental group
of an aspherical $n$-manifold has cohomological dimension $\le n$
with equality precisely when the manifold is 
closed~\cite[VIII.8.1]{Bro-book}. More generally,
an infinite index subgroup of the $n$-dimensional
Poincar{\'e} duality group has cohomological dimension 
$<n$~\cite{Str-PD}. 
\item
$G$ is a countable
group of finite cohomological
dimension if and only if there is 
a $K(G,1)$ manifold whose universal cover
is a Euclidean space.
(Find a finite-dimensional $K(G,1)$~\cite[VIII.7.1]{Bro-book},
make it a locally finite and of the same 
dimension~\cite[Theorem 13]{Wh-comb-top-I},
use simplicial approximation to make the space polyhedral,
embed it into a Euclidean space~\cite[Theorem 3.2.9]{Spa-book}, 
take the open regular neighborhood, and note that the universal cover
of any codimension $\ge 3$ open regular neighborhood
is simply-connected at infinity, as mentioned e.g. 
in the statement of~\cite[Theorem 2.1]{Sie-collar}.
Alternatively, one can take the product of the open regular neighborhood with $\mathbb R$, 
see Remark~\ref{rmk: prod of contr mnflds}.)
\end{enumerate}

\section{Homology spheres bounding contractible manifolds}
\label{sec: homology sph}

Results of Section~\ref{sec: without R-factors} require
a certain infinite family of compact contractible manifolds whose construction is reviewed below.

The problem which homology $3$-spheres bound 
contractible manifolds is not well understood.
(Freedman famously showed that all homology $3$-spheres bound topological contractible $4$-manifolds, but 
in this paper we only consider smooth manifolds).
An infinite family of homology $3$-spheres that bound contractible
manifolds was found by Casson-Harer~\cite{CasHar81} among  
Brieskorn spheres $\Sigma(p,q,r)$
with $\frac{1}{p}+\frac{1}{q}+\frac{1}{r}<1$, which are closed aspherical $3$-manifolds
modelled on $\widetilde{SL}_2(\mathbb R)$ 
geometry~\cite{Mil75}. For example, 
Casson-Harer's list contains the family
$\Sigma(p, p+1, p+2)$ where $p$ is any odd integer $\ge 3$.
The groups $\pi_1(\Sigma(p, p+1, p+2)\times S^1)$
are pairwise non-isomorphic because their quotients
by the centers is the fundamental groups
of $2$-orbifolds with Euler characteristics
$\frac{1}{p}+\frac{1}{p+1}+\frac{1}{p+2}-1$, which are
all distinct.

One can also construct an infinite family
of hyperbolic homology $3$-spheres that bound contractible manifolds, see the MathOverflow answers~\cite{AM} for details, 
but the above Casson-Harer's family is enough for our purposes.

Kervaire shows in~\cite[Theorems 1 and 3]{Ker69} 
that for each $m>3$
the fundamental group of any homology $3$-sphere is also the fundamental group of a 
homology $m$-sphere that bounds a contractible manifold. (Even though~\cite[Theorems 1]{Ker69}
is not stated for $m=4$, it still applies
in our case because as Kervaire remarks on~\cite[pp. 67-68]{Ker69}
the fundamental group of any homology $3$-sphere has a
presentations with an equal number of relations and generators).

\small
\bibliographystyle{amsalpha}

\def\cprime{$'$}
\providecommand{\bysame}{\leavevmode\hbox to3em{\hrulefill}\thinspace}
\providecommand{\MR}{\relax\ifhmode\unskip\space\fi MR }
\providecommand{\MRhref}[2]{%
  \href{http://www.ams.org/mathscinet-getitem?mr=#1}{#2}
}
\providecommand{\href}[2]{#2}

\end{document}